\documentclass[12pt, twoside]{article}
\usepackage{verbatim, version}
\usepackage{amsmath,amssymb,latexsym, amsthm}
\usepackage{graphicx, subfigure, epsfig, psfrag}
\usepackage{amsfonts, amssymb, dsfont, mathrsfs, color}
\usepackage{diagbox}
\usepackage{paralist,bbding,pifont}
\usepackage{hyperref}

\newlength{\tmarg}
\newlength{\bmarg}
\newlength{\lmarg}
\newlength{\rmarg}

\newtheorem{lemma}{Lemma}[section]
\newtheorem{remark}{Remark}[section]

\newtheorem{proposition}{Proposition}[section]

\newtheorem{theorem}{Theorem}[section]

\begin{document}
\include{macros}
\title{Weak solutions to the steady incompressible Euler equations with source terms}

\author{
  Anxiang Huang\thanks{School of Mathematical Sciences and Institute of Natural Sciences, Shanghai Jiao Tong University, Minhang, Shanghai, 200240, P. R. China ({\tt huanganx@sjtu.edu.cn}).}
}

%\date{\today}
\date{}

\maketitle
%------------------------------------------------------------------------------ 

%============================================================================== 

\begin{abstract}
  In this paper, we prove the non-uniqueness of stationary solutions to steady incompressible Euler equations with source terms. Based on the convex integration scheme developed by De Lellis and Sz\'{e}kelyhidi, the Euler system is reformulated as a differential inclusion. The key point is to construct the corresponding plane-wave solutions via high frequency perturbations. Then we use iteration and Baire category argument to conclude that there exist a large amount of weak solutions with given energy profile.
\end{abstract}


\pagestyle{myheadings}

\thispagestyle{plain}

\markboth{}{}

%------------------------------------------------------------------------------ 

%============================================================================== 

\section{Introduction}

\quad Consider the following $d$-dimensional steady incompressible Euler system with source terms:
\begin{align}\label{eq1}
\begin{cases}
    \text{div}(v\otimes v)+\nabla p=\mathbf{B}v,\\
    \text{div }v=0,
\end{cases}
\end{align}
where $x\in\Omega$ with $\Omega\subset\mathbb{R}^d$ or $\mathbb{T}^d$($d\geq 2$), $v$ and $p$ denote the velocity and the pressure of the flows respectively, and $\mathbf{B}$ is a $d\times d$ constant matrix.

\quad The Euler system expresses the conservation laws of mass and momentum. The well-posedness to the Euler system has always been a concern. The local well-posedness for smooth solutions to the Euler system is established in \cite{MR1867882}, and the global well-posedness is known to hold under the assumption that the $L^\infty$-norm of the velocity is $L^1$-integrable in time (\cite{MR0763762}). However, for smooth solutions to three-dimensional Euler equations, the global well-possedness is a challenging open problem. For the recent progress, one may refer to Hou-Luo (\cite{MR3176355, MR3278833}) that three-dimensional Euler equations may develop a singularity. 

\quad In the modern theory of partial differential equations, studies were more concerned with the weak solutions to \eqref{eq1}. Recall that $v\in L^\infty(\Omega)$ is a weak solution to \eqref{eq1} if
\begin{align*}
    \begin{cases}
    \int (v\otimes v\colon\nabla \phi+\phi\cdot \mathbf{B}v)dx=0,\\
    \int (v\cdot\nabla\psi)dx=0
    \end{cases}
\end{align*}
for every divergence-free vector field $\phi\in C_c^\infty(\Omega)$ and every scalar test function $\psi\in C_c^\infty(\Omega)$.

\quad Within the class of weak solutions, the Euler equations are known to have strange behavior. In particular, weak solutions might dissipate the total kinetic energy which is called "anomalous dissipation". The existence of dissipative solutions is first considered by Onsager in 1949 \cite{MR0036116}, where he conjectured that
\begin{itemize}
    \item {Any weak solution $v$ belonging to the H\"{o}lder space $C_{x,t}^{\theta}$ for $\theta>1/3$ conserves kinetic energy;}
    \item {For any $\theta<1/3$ there exist weak solutions $v\in C_{x,t}^{\theta}$ which dissipate kinetic energy.}
\end{itemize}

\quad The first part of the Onsager's conjecture was partially proved by Eyink (\cite{MR1302409}), and later fully proved by Constantin, E and Titi (\cite{MR1298949}). For the second part of the Onsager's conjecture, the first result traces back to Scheffer (\cite{MR1231007}), who constructed a nontrivial weak solution of the incompressible Euler equations in $\mathbb{R}^2\times\mathbb{R}$ with compact support in space-time. Strictly speaking, these weak solutions are not dissipative, as dissipative solutions are required to have non-increasing energy. A different construction of the existence of a compactly supported nontrivial weak solution in $\mathbb{T}^2\times\mathbb{R}$ was given by Shnirelman (\cite{MR1476315}). The first proof of the existence of a solution with monotone decreasing total kinetic energy was given by Shnirelman (\cite{MR1777341}). In the groundbreaking papers\cite{MR2600877,MR2564474}, De Lellis and Sz\'{e}kelyhidi showed the existence of infinitely many bounded weak solutions to the incompressible Euler solutions, yielding an alternative proof of Scheffer's non-uniqueness results. This breakthrough provides a better understanding of turbulence, and gives a new approach to proof the Onsager's conjecture. After that, a series of studies 
(\cite{MR3374958,MR3330471,MR3090182,MR3254331}), which mainly based on the same method, proved the existence and non-uniqueness of energy disspative solutions with H\"{o}lder exponent $\theta<1/3$. Finally, the remaining part of the proof to the Onsager's conjecture is accomplished in \cite{MR3896021,MR3866888}.

\quad Besides, the convex integration method also can be used to prove the existence and non-uniqueness of dissipate solutions to compressible Euler system (\cite{MR3352460,MR3261301,MR3269641}), steady incompressible Euler system (\cite{MR3505175}), and other fluid dynamic systems such as Navier-Stokes equations (\cite{MR4610908,MR3898708,MR3843425,MR4462623,MR3951691,MR4092687}), MHD equations (\cite{MR4510493,MR4328510,MR4510493,MR4198715}), SQG equations (\cite{MR3987721,MR4260790}), active scalar equations (\cite{MR2813340}) and so on (\cite{MR3884855,MR3884855,MR4138227}). Most of these studies consider about the case when $\mathbf{B}=\mathbf{0}$. When the source terms are non-trivial, the non-uniqueness of $L^\infty$ solutions to the compressible Euler equations has been proved in \cite{MR3459023}. 

\quad In this paper, we consider incompressible Euler equations with source terms. We show that under periodic boundary condition, the non-uniqueness and the analogue of the h-principle for stationary weak solutions in $L^\infty$ still holds even when the source terms are nontrivial. Moreover, without the periodic condition, we show the existence of infinitely many bounded weak solutions with compact supports. Our main results are as follows:

\begin{theorem}\label{tm1.1}
    Let $d\geq 2$,
    \begin{itemize}
        \item[(a)]{let $\Omega=\mathbb{T}^d$, $v_0$ be a smooth stationary flow to \eqref{eq1} on $\mathbb{T}^d$. Given a smooth function $e(x)>\left|v_0(x)\right|^2$ for $x\in \mathbb{T}^d$. Then, for any $\sigma>0$, there exist infinitely many weak solutions $v\in L^\infty(\mathbb{T}^d;\mathbb{R}^d)$ to \eqref{eq1} such that $\left|v(x)\right|^2=e(x)$ and $\left|v-v_0\right|_{H^{-1}}\leq \sigma$;}
        \item[(b)]{let $\Omega\subset\mathbb{R}^d$ be a bounded open set. Assume $e(x)\in C_c^\infty(\Omega;\overline{\mathbb{R}_{+}})$ is a non-trivial function, then there exist infinitely many compact supported weak solutions $v\in L^\infty(\mathbb{R}^d;\mathbb{R}^d)$ to \eqref{eq1} such that $\left|v(x)\right|^2=e(x)$.}
    \end{itemize}
\end{theorem}

\begin{remark}
    Compared with the results obtained by Choffrut and Sz\'{e}kelyhidi (\cite{MR3505175}), Theorem \ref{tm1.1}(a) shows that, when the source terms are nontrivial, there still exist a large amount of weak solutions in the neighbourhood of any smooth solution.
\end{remark}

\quad The proof is mainly based on the convex integration framework for the Euler equation. In Section 2, we reformulate the Euler equation as a differential inclusion. By using Tartar's framework, we can reformulate the nonlinear system as a linear system with nonlinear constraints. In Section 3, localized solutions are constructed, which is the building block to the iteration scheme. And in Sections 4 and 5, we construct suitable constraint sets which ensure the existence of the plane-wave solutions. Finally, by using Baire category argument, we prove the existence and non-uniqueness of the solutions to the Euler system.  

\section{Reformulation As A Differential Inclusion}

\quad First of all, recall the Tartar's framework for nonlinear systems. Consider the nonlinear partial differential equations which can be expressed as a system of linear partial differential equations
\begin{align}\label{eq2}
    \mathop{\sum}\limits_{i=1}^{m}A_i\partial_i z=0
\end{align}
with corresponding nonlinear constraint
\begin{align*}
    z(x)\in K\subset\mathbb{R}^d\quad \text{a.e.},
\end{align*}
where $z:\Omega\subset\mathbb{R}^m\rightarrow \mathbb{R}^d$ is the unknown state variable. Define the plane-wave solutions to \eqref{eq2} as the solutions of the form
\begin{align}\label{eq3}
    z(x)=ah(x\cdot \xi),
\end{align}
where $h:\mathbb{R}\rightarrow \mathbb{R}$ is a smooth profile. Define the wave cone $\Lambda$ to \eqref{eq2} as the set of states $a\in\mathbb{R}^d$ such that for any choice of profile function $h$, the function $z$ defined in \eqref{eq3} is a solution to \eqref{eq2}. By direct computation, the wave cone $\Lambda$ can be expressed as
\begin{align}\label{eq4}
    \Lambda:=\left\{a\in\mathbb{R}^d: \text{ there exists a }\xi\in\mathbb{R}^m\backslash \{0\}\quad\text{s.t.}\quad\mathop{\sum}\limits_{i=1}^m\xi_iA_ia=0\right\}.
\end{align}
\quad Under this framework, one can reformulate the nonlinear system to a linear system with nonlinear constraints. Since the plane wave solutions can be added linearly, the oscillatory behavior of solutions to the nonlinear system is then determined by the nonlinear constraints, that is the compatibility of the set $K$ with the corresponding wave cone $\Lambda$.

\quad The steady Euler system with source terms can then be reformulated as a differential inclusion under Tartar's framework.
Set $\mathbb{R}^m=\mathbb{R}^d$ and the state variable $z:=(v,\mathbf{U})$ with
\begin{align*}
    \mathbf{U}=v\otimes v-\frac{1}{d}\left|v \right|^2\mathbf{I}_d,
\end{align*}
where $\mathbf{I}_d$ denotes $d\times d$ identity matrix, and set the corresponding function $q=p+\left|v\right|^2/d$. Denote the set of $d\times d$ symmetric, trace-free matrices by
\begin{align*}
    \mathbb{S}_0^{d\times d}=\{\mathbf{U}\in\mathbb{R}^{d\times d}:\mathbf{U}^T=\mathbf{U}, \text{tr}(\mathbf{U})=0\}.
\end{align*}
By definition, $\mathbb{S}_0^{d\times d}$ is a linear subspace of $\mathbb{R}^{d_\ast}$ with $d_\ast=d(d+1)/2-1$. Denote the operator norm of $\mathbf{U}\in\mathbb{S}_0^{d\times d}$ by $\left|\mathbf{U}\right|$. For convenience, we assume that $\Omega=\mathbb{R}^d \text{ or } \mathbb{T}^d$. By direct computation, we have the following lemma.

\begin{lemma}\label{lm2.1}
    Let $d\geq 2$, let $e(x)\in C(\Omega)$ be a positive function. Assume that $(v,\mathbf{U},q)\in L^\infty(\Omega;\mathbb{R}^d)\times L^\infty(\Omega;\mathbb{S}_0^{d\times d})\times \mathcal{D}'(\Omega;\mathbb{R})$ is a weak solution to
\begin{align}\label{eq5}
\begin{cases}
	\text{div}\ \mathbf{U}+\nabla q=\mathbf{B}v,\\
	\text{div}\ v=0,
\end{cases}
\end{align}
in the sense of distribution. If
\begin{align*}
	\mathbf{U}=v\otimes v-\dfrac{e}{d}\mathbf{I}_d,\quad \text{a.e.}\quad x\in\Omega,
\end{align*}
then $(v,p)$ is a weak solution of \eqref{eq1}, where $p=q-e/d$ and $\left|v(x)\right|^2=e(x)$ for  $a.e.\ x\in\Omega$.
\end{lemma}

\quad We can now define the subsolution to the Euler system. We call a pair $w:=(v,\mathbf{U})\in\Omega\rightarrow \mathbb{R}^d\times \mathbb{S}^{d\times d}_0$ a stationary subsolution, if there exists a distribution $q\in\mathcal{D}'(\Omega)$ such that the triple $(v,\mathbf{U},q)$ is a weak solution to \eqref{eq5}. One important thing is that, though the source term of the Euler system may be non-trivial, this source term is linear under Tartar's framework. Hence, it only appears and makes differences in the plane-wave solutions. The nonlinear constraints are the same as the system without source terms.

\quad In order to construct solutions with specific energy profile, the nonlinear constraint set is defined as follows. For any $r>0$, denote
\begin{align}\label{eq6}
	K_r:=\left\{(v,\mathbf{U})\in\mathbb{R}^d\times \mathbb{S}_0^{d\times d}:\mathbf{U}=v\otimes v-\dfrac{r}{d}\mathbf{I}_d\right\}\subset \mathbb{R}^{d_{\ast}},
\end{align}
where $d_{\ast}=d(d+1)/2-1$. Since for any $r>0$, the set $K_r$ is a compact, smooth submanifold of $\mathbb{R}^{d_{\ast}}$ with dimension $d$, a weak solution to Euler equations \eqref{eq1} with energy profile $e(x)$ is therefore a subsolution $w=(v,\mathbf{U})$ which satisfies the pointwise inclusion
\begin{align*}
	w(x)\in K_{e(x)} \quad \text{for a.e.}\quad x\in\Omega.
\end{align*}
The idea of convex integration is then to relax the constraint set $K_{e(x)}$ to a suitable non-empty open subset of the convex hull:
\begin{align*}
	\mathcal{U}_{e(x)}\subset K_{e(x)}^\text{co}.
\end{align*}
\quad The key property required of the sets $\mathcal{U}_r\subset K_{r}^{\text{co}}$ is the following:

\textbf{Perturbation Property(P):} There is a strictly increasing function $\Phi: [0,+\infty)\rightarrow[0,+\infty)$ with $\Phi(0)=0$ and satisfies the following properties. Let $Q=(0,1)^{d}$ be the open unit cube in $\mathbb{R}^{d}$, then for every $\bar{w}=(\bar{v},\bar{\mathbf{U}})\in\mathcal{U}_{r}$, there exists a subsolution $w=(v,\mathbf{U})\in C_{c}^{\infty}(Q;\mathbb{R}^{d}\times \mathbb{S}_0^{d\times d})$ with associated pressure $q\in C_c^{\infty}(Q)$ such that
\begin{itemize}
	\item {$\bar{w}+w(x)\in\mathcal{U}_r$ for all $x\in Q$;}
	\item {$\int_Q\left|w(x)\right|^2dx\geq \Phi(\text{dist}(\bar{w},K_r))$;}
	\item {$\int_Qw(x)dx=0$.}
\end{itemize}

\quad Recall that, for the constraint set $K_r$ defined in \eqref{eq6}, the convex hull of $K_r$ can be explicitly formulated as
\begin{align}\label{eq7}
	K_r^{\text{co}}=\left\{(v,\mathbf{U})\in\mathbb{R}^{d}\times\mathcal{S}_0^{d\times d}:v\otimes v-\mathbf{U}\leq\dfrac{r}{d}\mathbf{I}_d\right\}.
\end{align}
Thus, for given $\bar{w}=(\bar{v},\bar{\mathbf{U}})\in K_r^{\text{co}}$, $\left|\bar{v}\right|^2=r$ implies $\bar{w}\in K_r$. Hence, there exists a continuous strictly increasing function $\Psi:[0,+\infty)\rightarrow[0,+\infty)$ with $\Psi(0)=0$ such that
\begin{align*}
	\text{dist}(\bar{w},K_r)\leq\Psi(r-\left|\bar{v}\right|^2)\quad \text{for all}\quad \bar{w}=(\bar{v},\bar{\mathbf{U}})\in K_r^{\text{co}}.
\end{align*}
One can then replace $\text{dist}(\bar{w},K_r)$ by $r-\left|\bar{v}\right|^2$ in Property (P).

\quad In order to obtain the precise statement of the Theorem \ref{tm1.1}, in addition to the Property (P) we require:
\begin{align}\label{eq8}
	K_r\subset \mathcal{U}_{r'}\quad\text{for}\quad r<r'.
\end{align}
Property \eqref{eq8} will ensure that smooth stationary flows belong to the set of subsolutions given by the relaxed set $\mathcal{U}_r$.

\section{Localized Plane-waves}

\quad In this section, the wave cone and the corresponding plane-waves are constructed, which form the building blocks for the iteration scheme. 

\quad  First, consider the stationary Euler system:
\begin{align*}
    \begin{cases}
        \text{div} (v\otimes v)+\nabla p=0,\\
        \text{div }v=0,
    \end{cases}
\end{align*}
and the corresponding differential inclusion is of the form
\begin{align}\label{eq9}
    \begin{cases}
        \text{div }\mathbf{U}+\nabla q=0,\\
        \text{div }v=0.
    \end{cases}
\end{align}
\quad With the form of differential inclusion and the corresponding nonlinear constraint set, one can then determine the corresponding wave cone to the linear system.

\quad Consider the $(d+1)\times d$ matrix in block form
\begin{align*}
    \bar{\mathbf{U}}'=\begin{pmatrix}
        \mathbf{U}+q\mathbf{I}_d\\
        v
      \end{pmatrix}.
\end{align*}
Then, the system \eqref{eq9} becomes the form $\text{div }\mathbf{U}'=0$. Substituting the plane-wave solutions \eqref{eq3} to the linear system \eqref{eq9} gives
\begin{align*}
    \text{div }\bar{\mathbf{U}}'=\text{div}\left[\begin{pmatrix}
    \mathbf{U}+q\mathbf{I}_d\\
    v
    \end{pmatrix}h(x\cdot \xi)\right]
    =\begin{pmatrix}
    \mathbf{U}+q\mathbf{I}_d\\
    v
    \end{pmatrix}\cdot \xi h'(x\cdot\xi)=0.
\end{align*}
By definition of the wave cone, $z(x)=(v,\mathbf{U})h(x\cdot \xi)$ are plane-wave solutions for any choice of profile function $h$, hence the wave cone corresponding to system \eqref{eq9} can be expressed as
\begin{align}\label{eq10}
	\Lambda=\left\{(\bar{v},\bar{\mathbf{U}})\in(\mathbb{R}^{d}\backslash\{0\})\times\mathbb{S}_0^{d\times d}:\exists \bar{q}\in\mathbb{R},\xi\in\mathbb{R}^d\backslash\{0\} \text{ s.t } \bar{\mathbf{U}}\xi+\bar{q}\xi=0,\ \bar{v}\cdot \xi=0\right\}.
\end{align}

\quad In order to construct solutions to the nonlinear system, one needs to find plane wave like solutions which are localized in space. In fact, the exact plane-wave solutions with compact supports are identically zero. Therefore, for given wave state $a\in\mathbb{R}^d$, one needs to introduce an error and construct the plane-wave solutions to the nonlinear system such that this error is under control. For the incompressible Euler equations, the set of the wave-cone \eqref{eq10} corresponds to the plane-wave solutions to \eqref{eq9}. Hence, one can construct stationary subsolutions with specific oscillatory behavior. Recall that, if the source term is trivial, i.e. $\mathbf{B}=\mathbf{0}$, one can construct the localize the plane-wave solutions by using the same potentials as in the time-dependent case. More precisely, we have the following lemma.

\begin{lemma}(\cite[Proposition 3.2]{MR2600877} and \cite[Proposition 20]{MR2968597})\label{lm3.1}
Let $d\geq 2$, given $\bar{w}=(\bar{v},\bar{\mathbf{U}})\in \Lambda$, then
\begin{enumerate}
    \item [(a)]{there exists an $\eta\in\mathbb{R}^d\backslash\{0\}$ such that for any $h\in C^\infty(\mathbb{R}),$\begin{align*}
        w(x):=\bar{w}h(x\cdot\eta)
    \end{align*}is a subsolution;}
    \item [(b)]{there exists a second order homogeneous linear differential operator $\mathcal{L}_{\bar{w}}$ such that\begin{align*}
        w:=\mathcal{L}_{\bar{w}}[\phi]
    \end{align*} is a subsolution for any $\phi\in C^\infty(\mathbb{R}^d)$;}
    \item [(c)]{moreover, if $\phi(x)=H(x\cdot \eta)$ for some $H\in C^\infty(\mathbb{R})$, then\begin{align*}
        \mathcal{L}_{\bar{w}}[\phi](x)=\bar{w}H''(x\cdot \eta).
    \end{align*}}
\end{enumerate}
\end{lemma}

\quad However, when the source terms are nontrivial, for $w_1,w_2\in K_{e(x)}$, there may not be plane wave solutions to \eqref{eq5} with profiles $w_1-w_2$. But we note that, in the high-frequency regime, localized plane waves with sources can be constructed as perturbations of the plane waves of the homogeneous system.

\quad In order to compute the errors from the source terms, we need the following lemma.
\begin{lemma}\label{lm3.2}(\cite[Lemma 8]{MR3459023})
	For any $f\in C_c^{\infty}(\mathbb{R}^d;\mathbb{R}^d)$, there exists an $\mathcal{R}[f]\in\mathcal{C}^{\infty}(\mathbb{R}^d;\mathbb{S}_0^{d\times d})$ satisfying
	\begin{align*}
		\nabla \cdot \mathcal{R}[f]=f.
	\end{align*}
	Furthermore, $\mathcal{R}$ satisfies the following properties:
	\begin{itemize}[(i)]
		\item [(a)]{$\mathcal{R}$ is a linear operator;}
		\item [(b)]{$\mathcal{R}[\Delta^2f]$ is a linear combination of third order derivatives of $f$;}
		\item [(c)]{$\text{supp}\mathcal{R}[\Delta^2f]\subset \text{supp} f$ and $\mathcal{R}[\Delta^2f]\in L(\mathbb{R}^d)$ i.e.
			\begin{align*} 
				\int_{\mathbb{R}^d}\mathcal{R}[\Delta^2f]dx=0;
			\end{align*}}
		\item [(d)]{there exists a constant $0<\alpha<1$ such that\begin{align*} \|\mathcal{R}[f]\|_{C^{\alpha}(\mathbb{R}^d)}\leq C\max(\|f\|_{\mathcal{L}^{1}(\mathbb{R}^d)},\|f\|_{\mathcal{L}^{\infty}(\mathbb{R}^d)}),\end{align*}
			
			where the constants $C$ and $\alpha$ depend only on dimension $d$.}
	\end{itemize}
\end{lemma}	

\quad By using Lemma \ref{lm3.1} and Lemma \ref{lm3.2}, one can construct the localized plane-wave solutions to \eqref{eq5}, which form the basic building-block to obtain more complicated oscillatory behavior.

\begin{lemma}\label{lm3.3}
Let $d\geq 2$, let $w,w_1,w_2\in\mathbb{R}^{d}\times \mathcal{S}_0^{d\times d}$, and $\mu_1,\mu_2>0$ be such that
\begin{align*}
	\bar{w}=w_2-w_1\in\Lambda,\quad w=\mu_1w_1+\mu_2w_2,\quad \mu_1+\mu_2=1.
\end{align*}
Given a bounded open set $\mathcal{O}\in\mathbb{R}^d$, and any $\varepsilon>0$, there exists a subsolution $\tilde{w}\in C_c^{\infty}(\mathcal{O};\mathbb{R}^{d}\times\mathbb{S}_0^{d\times d})$ such that
\begin{itemize}
	\item [(a)]{$\text{dist}(w(x)+\tilde{w},[w_1,w_2])<\varepsilon$ for all $x\in \mathcal{O}$;}
	\item [(b)]{there exist disjoint open subsets $\mathcal{O}_1,\mathcal{O}_2\in\mathcal{O}$ such that for $i=1,2$,
	\begin{align*}
		\left|w+\tilde{w}-w_i\right|<\varepsilon\quad \text{for}\quad x\in \mathcal{O}_i,\quad \left|\mathcal{H}^d\left|\mathcal{O}_i\right|-\mu_i\mathcal{H}^d(\mathcal{O})\right|<\varepsilon.
	\end{align*}}
\end{itemize}
\end{lemma}

\begin{proof}
	Let $\delta$ be a small positive constant and $\phi\in C_c^{\infty}$ be a smooth cut-off function satisfying
	\begin{align*}
		0\leq\phi\leq 1\quad \text{and}\quad \mathcal{H}^d\left(\{x\in\mathcal{O}:\phi(x)\neq 1\}\right)<\delta,
	\end{align*}
	and let
	\begin{align*}
		h(s)=\begin{cases}
			-\mu_2,\ \ \ &s\in(0,\mu_1],\\
			\mu_1,\ \ \ &s\in(\mu_1,1].
		\end{cases}
	\end{align*}
	Clearly, $\int_0^1h(s)ds=-\mu_2\cdot \mu_1+\mu_1\cdot(1-\mu_1)=0$. 
	
	\quad Extend $h$ as a periodic function with period $1$ and let $h_0$ be a smooth approximation of $h$ satisfying
	\begin{align*}
		-\mu_2\leq h_0\leq \mu_1,\quad \mathcal{H}^1\left(\{s\in[0,1]:h(s)\neq h_0(s)\}\right)<\delta,\quad \int_0^1h_0(s)ds=0.
	\end{align*}
	Define $h_k$ by induction as
	\begin{align*}
		\tilde{h}_{k+1}(s)=\int_0^sh_k(t)dt,\quad h_{k+1}(s)=\tilde{h}_{k+1}(s)-\int_0^1\tilde{h}_{k+1}(t)dt.
	\end{align*}
	Hence, for $k\geq 0$, we have
	\begin{align*}
		\dfrac{d^jh_k}{ds^j}=h_{k-j}, \quad \int_0^1h_k(s)ds=\int_0^1\left(\tilde{h}_k(s)-\int_0^1\tilde{h}_k(t)dt\right)ds=0,
	\end{align*}
	and
	\begin{align*}
		\|h_k\|_{L^{\infty}}\leq\|h_{k-1}\|_{L^{\infty}}\leq\cdots\leq\|h_0\|_{L^{\infty}}.
	\end{align*}
	\quad Since $\bar{w}=w_2-w_1=\left(\bar{v},\bar{\mathbf{U}}\right)\in\Lambda$, by Lemma \ref{lm3.1}, there exists a $\xi\in\mathbb{R}^d\backslash \{0\}$ such that $\bar(w)h(x\cdot \xi)$ is a subsolution. Let $\bar{q}$ be the corresponding function such that $\bar{\mathbf{U}}\eta+\bar{q}\eta=0$. For a given large constant $\lambda\in\mathbb{R}$, let
	\begin{align*}
		v'=\lambda^{-6}\Delta^3[\bar{v}h_6(\lambda\xi\cdot x)\phi],\quad \mathbf{U}'=\lambda^{-6}\Delta^3[\bar{\mathbf{U}}h_6(\lambda\xi\cdot x)\phi],\quad
		q'=\lambda^{-6}\Delta^3[\bar{q}h_6(\lambda\xi\cdot x)\phi],
	\end{align*}
	and
	\begin{align*}
		v''=-\nabla\Delta^{-1}\nabla \cdot v',\quad \mathbf{U}''=\mathcal{R}[\mathbf{B}(v'+v'')-\nabla q'-\nabla\cdot \mathbf{U}'].
	\end{align*}
	\quad Define
	\begin{align*}
		w=(v,\mathbf{U})=(v'+v'',\mathbf{U}'+\mathbf{U}'').
	\end{align*}
	  It suffices to show that $w$ is the function which satisfies the properties of this lemma.
	
	\quad Direct computation shows that
	\begin{align*}
		\nabla\cdot v=\nabla\cdot(v'+v'')=\nabla\cdot v'-\nabla\cdot\nabla\Delta^{-1}\nabla\cdot v'=0,
	\end{align*}
	and
	\begin{align*}
		\nabla\cdot \mathbf{U}=\nabla \cdot \mathbf{U}'+\nabla\cdot\mathcal{R}[\mathbf{B}v-\nabla q'-\nabla\cdot \mathbf{U}']=\nabla \cdot \mathbf{U}'+\mathbf{B}v-\nabla q'-\nabla\cdot \mathbf{U}'=\mathbf{B}v-\nabla q'.
	\end{align*}

	Hence, $(v,\mathbf{U})$ is a subsolution to the system \eqref{eq1} with $q=q'=\lambda^{-6}\Delta^3[\bar{q}h_6(\lambda\xi\cdot x)\phi]$. Furthermore,
	\begin{align*}
		v'&=\bar{v}h_0(\lambda\xi\cdot x)\phi+\lambda^{-6}\bar{v}\underset{\left|\beta\right|\geq 1,\left|\alpha+\beta\right|=6}{\Sigma}C_{\alpha,\beta}\partial_x^{\alpha}h_6(\lambda\xi\cdot x)\partial_x^{\beta}\phi\notag\\
		&=\bar{v}h_0(\lambda\xi\cdot x)\phi+\lambda^{-6}\bar{v}\underset{\left|\beta\right|\geq 1,\left|\alpha+\beta\right|=6}{\Sigma}\lambda^{\left|\alpha\right|}C_{\alpha,\beta}h_{6-\left|\alpha\right|}(\lambda\xi\cdot x)\partial_x^{\beta}\phi,
	\end{align*}
	and
	\begin{align*}
		\mathbf{U}'=\bar{\mathbf{U}}h_0(\lambda\xi\cdot x)\phi+\lambda^{-6}\bar{\mathbf{U}}\underset{\left|\beta\right|\geq 1,\left|\alpha+\beta\right|=6}{\Sigma}\lambda^{\left|\alpha\right|}C_{\alpha,\beta}\xi^{\alpha}h_{6-\left|\alpha\right|}(\lambda\xi\cdot x)\partial_x^{\beta}\phi.
	\end{align*}
	Since $\text{supp}\ (n',\mathbf{U}')\subset \text{supp}\ \phi$, and $\int(v',\mathbf{U}')dx=0$, then one has
	\begin{align*}
		\|(v',\mathbf{U}')-(\bar{v},\bar{\mathbf{U}})h_0(\lambda\xi\cdot x)\phi\|_{\mathcal{L}^{\infty}}\leq C(\left|\bar{v}\right|,\left|\bar{\mathbf{U}}\right|,\phi,h_0)\lambda^{-1}.
	\end{align*}
	It follows from $\bar{v}\cdot \xi=0$ that one has
	\begin{align*}
		v''=-\lambda^{-6}\nabla\Delta^{-1}\Delta^{3}\nabla\cdot[\bar{v}h_6(\lambda\xi\cdot x)\phi]=-\lambda^{-6}\nabla\Delta^2[(\bar{v}\cdot \nabla \phi)h_6(\lambda\xi\cdot x)].
	\end{align*}
	Therefore, $\text{supp}\ v''\subset \text{supp}\ \phi$, $\int v''dx=0$ and
	\begin{align*}
		\|v''\|_{L^{\infty}}\leq C(\left|\bar{v}\right|,\phi,h_0)\lambda^{-1}.
	\end{align*}
	As for $\mathbf{U}''$, since $\bar{\mathbf{U}}\cdot \xi+\bar{q}\xi=0$, one can compute directly that
	\begin{align*}
		\mathbf{U}''=&-\mathcal{R}[\nabla\cdot \mathbf{U}'-\mathbf{B}v+\nabla q']\\
		=&-\lambda^{-6}\mathcal{R}\{\Delta^3[\nabla\cdot(\bar{\mathbf{U}}h_6(\lambda\xi\cdot x)\phi)]-\mathbf{B}\Delta^3[\bar{v}h_6(\lambda\xi\cdot x)\phi]\\
            &+\mathbf{B}\nabla\Delta^2[(\nabla\phi\cdot\bar{v})h_6(\lambda\xi\cdot x)]+\Delta^3[\nabla\cdot(\bar{q}h_6(\lambda\xi\cdot x)\phi \mathbf{I}_d)]\}\\
		=&-\lambda^{-6}\mathcal{R}\{\Delta^3[\bar{\mathbf{U}}\cdot \nabla\phi h_6(\lambda \xi\cdot x)]-\Delta^2\mathbf{B}[\nabla((\nabla\phi\cdot\bar{v})h_6(\lambda\xi\cdot x))\\
            &-\Delta(\bar{v}h_6(\lambda\xi\cdot x)\phi)]+\Delta^3[\bar{q}\cdot\nabla\phi h_6(\lambda\xi\cdot x)\mathbf{I}_d]\}.
	\end{align*}
	Then,
	\begin{align*}
		\mathbf{U}''=-\lambda^{-6}\mathcal{R}[\Delta^2f],
	\end{align*}
	where
	\begin{align*}
		f=\Delta[\bar{\mathbf{U}}\cdot \nabla\phi h_6(\lambda \xi\cdot x)]-\mathbf{B}[\nabla((\nabla\phi\cdot\bar{v})h_6(\lambda\xi\cdot x))-\Delta(\bar{v}h_6(\lambda\xi\cdot x)\phi)]+\Delta[\bar{q}\cdot\nabla\phi h_6(\lambda\xi\cdot x)\mathbf{I}_d].
	\end{align*}
	Hence, $\text{supp }\mathbf{U}''\subset \text{supp } f\subset \text{supp } \phi$, and $\int_{\mathbb{R}^n}\mathbf{U}''dx=0$. Furthermore, 
	\begin{align*}
		\|\mathbf{U}''\|_{L^{\infty}(\mathcal{Q})}=\lambda^{-6}\|\mathcal{R}[\Delta^2f]\|_{L^{\infty}}\leq C(\left|\bar{v}\right|,\left|\bar{\mathbf{U}}\right|,\phi,h_0)\lambda^{-1}.
	\end{align*}
	\quad Then, we conclude that
	\begin{align*}
		\|w(x)-\bar{w}h_0(\lambda\xi\cdot x)\phi(x)\|_{L^{\infty}}\leq C(\bar{w},\phi,h_0)\lambda^{-1}.
	\end{align*}
	For $x\in Q$, it holds that $\bar{w}(h_0\phi)(x)\in[-\mu_2\bar{w},\mu_1\bar{w}]$. Thus,
	\begin{align*}
		\text{dist}(w,[w_1,w_2])&\leq \text{dist}(\bar{w}h_0\phi,[w_1,w_2])+\left|w-\bar{w}h_0(\lambda\xi\cdot x)\phi\right|\\
		&\leq C(\bar{w},\phi,h_0)\lambda^{-1}.
	\end{align*}
	Define the disjoint open sets $\mathcal{O}_i$ as
	\begin{align*}
		\mathcal{O}_i=\left\{x\in\mathcal{O}:\left|\bar{w}h_0(\lambda\xi\cdot x)\phi(x)-w_i\right|<\min\left(\dfrac{\varepsilon}{2},\dfrac{\left|w_2-w_1\right|}{4}\right)\right\}.
	\end{align*}
	
	Thus, for $x\in\mathcal{O}_i$,
	\begin{align*}
		\left|w(x)-w_i\right|&\leq\left|\bar{w}h_0(\lambda\xi\cdot x)\phi-w_i\right|+\left|w-\bar{w}h_0(\lambda\xi\cdot x)\phi\right|\\
		&\leq\dfrac{\varepsilon}{2}+C(\bar{w},\phi,h_0)\lambda^{-1}.
	\end{align*}
	
	Therefore, if $\delta$ is small enough and $\lambda$ is large enough, all the properties hold. This finishes the proof of the lemma.
\end{proof}

\quad With these localized plane waves, one can show that the wave cone has the following properties.

\begin{lemma}\label{lm3.4} Let $\Lambda$ be defined as \eqref{eq10}, $Q_1=(0,1)^d$ be the open unit cube. Then there exists a constant $C>0$ such that for any $(\bar{v},\bar{\mathbf{U}})\in\Lambda$, there exist a sequence of subsolutions $(v_k,\mathbf{U}_k)\in C_c^{\infty}(Q;\mathbb{R}^d\times \mathbb{S}_0^{d\times d})$ such that
\begin{enumerate}
    \item [(a)] {$\text{dist}((v_k,U_k),[-(\bar{v},\bar{\mathbf{U}}),(\bar{v},\bar{\mathbf{U}})])\rightarrow 0$ uniformly in $Q_1$ as $k\rightarrow \infty$;}
    \item [(b)] {$(v_k,U_k)\rightarrow 0$ in the sense of distribution as $k\rightarrow \infty$;}
    \item [(c)] {$\int_{Q_1}\left|v_k,U_k\right|^2(x)dx\geq C\left|(\bar{v},\bar{\mathbf{U}})\right|^2$ for all $k\in\mathbb{N}$.}
\end{enumerate}
\end{lemma}

\begin{proof}
    Let $\bar{w}=(\bar{v},\bar{\mathbf{U}})\in\Lambda$. For each $k\in\mathbb{N}$, let $w=0=-\bar{w}/2+\bar{w}/2$, $\varepsilon=2^{-kd}/k$, $\mathcal{O}=Q_{2^{-k}}$, applying Lemma \ref{lm3.3}, one can obtain a sequence of subsolutions $\tilde{w}_k\in C_c^{\infty}(Q_{2^{-k}})$ with corresponding $\tilde{q}_k$ satisfying
    \begin{align*}
        \text{dist}(\tilde{w}_k,[-\bar{w},\bar{w}])<\frac{2^{-kd}}{k}.
    \end{align*}
    Furthermore, there exist $\mathcal{O}_i^{(k)}$ such that
    \begin{align*}
        \left|\tilde{w}_k-(-1)^i\bar{w}\right|<\frac{2^{-kd}}{k} \text{ for}\ x\in \mathcal{O}_i,\ \left|\mathcal{H}^d(\mathcal{O}^{(k)}_i)-2^{-kd-1}\right|<\frac{2^{-kd}}{k}.
    \end{align*}
    \quad Let $\{Q_{2^{-k}}(x^{(j,k)})\}$ be the decomposition of the unit cube $Q_1$ into mutually disjoint cubes with length $2^{-k}$. Denote
    \begin{align*}
        w_k=\mathop{\sum}\limits_{j=1}^{2^{kd}}\tilde{w}_k(x-x^{(j,k)}),\ q_k=\mathop{\sum}\limits_{j=1}^{2^{kd}}\tilde{q}_k(x-x^{(j,k)}).
    \end{align*}
    Note that $\text{supp}(\tilde{q}_k)\in Q_{2^{-k}}(x^{(j,k)})$ since $\tilde{w}_k\in C_c^{\infty}(Q_{2^{-k}}(x^{(j,k)}))$ are subsolutions with corresponding $\tilde{q}_k$. Hence, $w_k$ is also a subsolution, and $\text{dist}(w_k,[-\bar{w},\bar{w}])\rightarrow 0$ as $k\rightarrow \infty$. Furthermore, since $\text{diam}(\text{supp }\tilde{w}_k)\leq \sqrt{n}2^{-k}\rightarrow 0$, then $w_k\rightarrow 0$ in the sense of distribution. By direct computations,
    \begin{align*}
        \int_{Q_1}\left|w_k\right|^2-\left|\bar{w}\right|^2&=\mathop{\sum}\limits_{j=1}^{2^{kd}}\int\left|w_k\right|^2-\left|\bar{w}\right|^2=2^{kd}\int_{Q_{2^{-k}}}\left|\tilde{w}_k\right|^2dx-\left|\bar{w}\right|^2\\
       &=\frac{1}{\left|Q_{2^{-k}}\right|}\int_{Q_{2^{-k}}}\left|\tilde{w}_k\right|^2-\left|\bar{w}\right|^2dx\\
       &=\frac{1}{\left|Q_{2^{-k}}\right|}\left(\mathop{\sum}\limits_{i=1}^{2}\int_{\mathcal{O}_i}\left|\tilde{w}_k\right|^2-\left|\bar{w}\right|dx+\int_{Q_{2^{-k}}\backslash(\mathcal{O}_1\cup\mathcal{O}_2)}\left|\tilde{w}_k\right|^2dx-\left|\bar{w}\right|^2dx\right)\\
       &<C\frac{1}{k}.
    \end{align*}
    Then, for sufficiently large $k$, it follows that
    \begin{align*}
        \int_{Q_1}\left|w_k\right|^2dx\geq\frac{1}{2}\left|\bar{w}\right|^2.
    \end{align*}
\end{proof}

\quad We have shown that, even when the source terms are nontrivial, there exist abundant localized plane waves in $\Lambda-$direction. In the next two sections, we aim to construct suitable relaxation sets for convex integration scheme. This part closely follows \cite{MR3505175}. We continue to use the construction of the relaxation set and prove the compatibility with the wave cone under the appearance of the nontrivial source terms. For more details, we refer to Sections 4 and 5 in \cite{MR3505175}.

\section {The case $d\geq 3$}
\quad For given $r>0$, we call a line segment $\sigma \subset \mathbb{R}^{d}\times\mathbb{S}_0^{d\times d}$ to be admissible if
\begin{itemize}
	\item [(a)]{$\sigma$ is contained in the interior of $K_{r}^\text{co}$;}
	\item [(b)]{$\sigma$ is parallel to $(a,a\otimes a)-(b,b\otimes b)$ for some $a,b\in\mathbb{R}^{d}$ with $\left|a\right|^2=\left|b\right|^2=r$ and $b\neq \pm a$.}
\end{itemize}

\quad First, we have the following lemma.

\begin{lemma}\label{lm4.1} Let $d\geq 2$, then there exists a constant $C=C(d,r)>0$ such that for any $w=(v,\mathbf{U})\in\mathring{K}_r^\text{co}$, there exists an admissible line segment $\sigma=[w-\bar{w},w+\bar{w}]$, $\bar{w}=(\bar{v},\bar{\mathbf{U}})$ such that
\begin{align*}
	\left|\bar{v}\right|\geq C(r-\left|v\right|^2)\quad \text{and}\quad \text{dist}(\sigma,\partial K_r^{\text{co}})\geq\dfrac{1}{2}\text{dist}(w,\partial K_r^{\text{co}}).
\end{align*}
\end{lemma}

\begin{proof}
Let $w=(v,\mathbf{U})\in\mathring{K}_r^{\text{co}}$, then, $(v,\mathbf{U})$ lies in the interior of a convex polytope of $\mathbb{R}^{d}\times\mathbb{S}_0^{d\times d}$ spanned by $N_{\ast}$ elements of $K_r$. Denote such elements by $w_i=(v_i,v_i\otimes v_i-\dfrac{\left|r\right|^2}{d}\mathbf{I}_d)$, where $v_i\in\mathbb{R}^{d}$ and $\left|v_i\right|=r$. Since $(v,u)$ is contained in the interior of the polytope, by possibly perturbing the $w_i$, we ensure that $v_i\neq \pm v_j$ whenever $i\neq j$.

\quad By Carath\'{e}odory's theorem, $w$ can be written as a positive convex combination of at most $N+1$ of the $w_i$, where $N=d(d+3)/2-1$, i.e.
\begin{align*}
	w=\mathop{\Sigma}\limits_{i=1}\limits^{k+1}\lambda_iw_i,
\end{align*}
where $\lambda_i\in(0,1)$, $\Sigma_{i=1}^{d+1}\lambda_i=1$, and $1\leq k\leq N$. Assume further that the coefficients are ordered so that $\lambda_1=\underset{i}{\max}\lambda_i$. For any $j>1$, consider the segment $\sigma_j=[w-\frac{\lambda_j}{2}(w_j-w_1),w+\frac{\lambda_j}{2}(w_j-w_1)]$, then
\begin{align*}
	\text{dist}(\sigma_j,\partial K_r^{\text{co}})\geq\dfrac{1}{2}\text{dist}(w,\partial K_r^{\text{co}}).
\end{align*}
In fact, if $B=B_{\rho}(w)\subset K_r^{\text{co}}$, then $K_r^{\text{co}}$ contains the convex hull of $B\cup\{w\pm\lambda_j(w_j-w_1)\}$, which obviously contains the open balls $B_{\rho/2}(w)$ for every $w\in\sigma_j$.

\quad On the other hand, $w-w_1=\mathop{\Sigma}\limits_{i=2}\limits^{k+1}\lambda_i(w_i-w_1)$, so that
\begin{align*}
	\left|v-v_1\right|\leq k\underset{i=2,3,...,k+1}{\max}\lambda_i\left|v_i-v_1\right|\leq N\underset{i=2,3,...,N+1}{\max}\lambda_i\left|v_i-v_1\right|.
\end{align*}
Let $j>1$ be such that $\lambda_j\left|v_i-v_1\right|=\underset{i=2,...,k+1}{\max}\lambda_i\left|v_i-v_1\right|$, and let
\begin{align*}
	(\bar{v},\bar{\mathbf{U}})=\dfrac{1}{2}\lambda_j(w_j-w_1)=\dfrac{1}{2}\lambda_j(v_j-v_1,v_j\otimes v_j-v_1\otimes v_1).
\end{align*}
Then, $\sigma=[(v,\mathbf{U})-(\bar{v},\bar{\mathbf{U}}),(v,\mathbf{U})+(\bar{v},\bar{\mathbf{U}})]$ is contained in the interior of $K_r^{\text{co}}$, hence it is an admissible segment. Moreover, by the choice of $j$, we have
\begin{align*}
	\dfrac{1}{4rN}(r^2-\left|v\right|^2)=\dfrac{1}{4rN}(r+\left|v\right|)(r-\left|v\right|)\leq\dfrac{1}{2N}\left|v-v_1\right|\leq\left|\bar{v}\right|.
\end{align*}
This finishes the proof of the lemma.
\end{proof}
\quad Furthermore, it should be noted that the admissible line segments are in $\Lambda-$directions in the case $d\geq 3$.

\begin{lemma}\label{lm4.2} Let $d\geq 3$, let $a,b\in\mathbb{R}^{d}$ be such that $\left|a\right|=\left|b\right|=r$ for some constant $r$, and $b\neq \pm a$, Let $(\bar{v},\bar{U})=(a,a\otimes a)-(b,b\otimes b)$. Then, $(\bar{v},\bar{\mathbf{U}})\in\Lambda$.
\end{lemma}
\begin{proof}
	Since $d\geq 3$, there exists a $\xi\in\mathbb{R}^d\backslash\{0\}$, satisfying $\xi\cdot a=\xi\cdot b=0$. Then, by directly computations, we have
	\begin{align*}
		\xi\cdot \bar{v}=\xi\cdot (b-a)=\xi\cdot b-\xi\cdot a=0,\\
		\bar{\mathbf{U}}\cdot \xi=(a\otimes a-b\otimes b)\cdot \xi=\xi\cdot a(a-b)=0.
	\end{align*} 
	Hence, $(\bar{v},\bar{\mathbf{U}})\in\Lambda$ with associated $q=0$.
\end{proof}

\begin{lemma}\label{lm4.3} Let $d\geq 3$, $\mathcal{U}_r:=\mathring{K}_r^{\text{co}}$ satisfies the perturbation Property (P) and Property \eqref{eq8}.
\end{lemma}

\begin{proof}
Note that
\begin{align*}
    K_r^{\text{co}}=\left\{(v,u)\in\mathbb{R}^{n}\times\mathbb{S}_0^{n\times n}:v\otimes v-\mathbf{U}\leq\dfrac{r}{d}\mathbf{I}_d\right\}.
\end{align*}
Hence, Property \eqref{eq8} holds.

\quad To show Property (P), let $\bar{w}\in\mathcal{U}_r$, then by Lemma \ref{lm4.1} and Lemma \ref{lm4.2}, we can find $\tilde{w}\in\Lambda$ such that
\begin{align*}
	[\bar{w}-\tilde{w},\bar{w}+\tilde{w}]\subset\mathcal{U}_r,\ \text{dist}([\bar{w}-\tilde{w},\bar{w}+\tilde{w}],\partial K_r^{\text{co}})\geq\dfrac{1}{2}\text{dist}(\bar{w},\partial K_r^{\text{co}}).
\end{align*}
Furthermore,
\begin{align*}
	\left|\tilde{v}\right|\geq\dfrac{1}{4}C(r-\left|\bar{v}\right|^2).
\end{align*}
\quad Applying Lemma \ref{lm3.4} with $\varepsilon<\dfrac{1}{4}\text{dist}(\bar{w},\partial K_r^{\text{co}})$, we can construct a subsolution $w=(v,\mathbf{U})\in C_c^{\infty}(Q;\mathbb{R}^{d}\times\mathbb{S}_0^{d\times d})$ such that $\bar{w}+w(x)\in\mathcal{U}_r$ for all $x\in Q$ and
\begin{align*}
	\int_Q\left|w(x)\right|^2dx\geq\dfrac{1}{2}\left|\tilde{w}\right|^2\geq C'(r-\left|\bar{v}\right|^2)^2
\end{align*}
for some constant $C'>0$. Hence the Property (P) holds.
\end{proof}

\section {The case $d=2$}
\quad Consider the case $d=2$. Note that in this case, we cannot use Lemma \ref{lm4.2}. In order to ensure the relaxation set is neither too large nor too small, we need to set up suitable coordinates and construct the relaxation set explicitly. 

\quad Note that, in the case $d=2$, the wave cone $\Lambda$ can be written as
\begin{align*}
	\Lambda=\left\{(\bar{v},\bar{\mathbf{U}})\in(\mathbb{R}^2\backslash\{0\})\times\mathbb{S}_0^{2\times 2}:\bar{\mathbf{U}}\bar{v}\cdot \bar{v}^{\perp}=0\right\}.
\end{align*}
In fact, $(\bar{v},\bar{\mathbf{U}})\in \Lambda$ if $\mathbf{U}$ possesses an eigenvector perpendicular to $v$. In two dimensions, this means that $\bar{v}^{\perp}$ is an eigenvector of $v$.

\quad We set up new coordinates on the state-space $\mathbb{R}^2\times \mathbb{S}_0^{2\times 2}$. The state variables $(v,\mathbf{U})\in\mathbb{R}^2\times \mathcal{S}_0^{2\times 2}$ can be rewritten as
\begin{align*}
v=\left(\begin{array}{rr}
a\\
b
\end{array}\right),\ \mathbf{U}=\left(\begin{array}{rrr}
c & d\\
d & -c\\
\end{array}\right).
\end{align*}
Let
\begin{align*}
    z=a+ib,\ \zeta=c+id.
\end{align*}
then, we can write
\begin{align*}
    w=(z,\zeta)\in\mathbb{R}^2\times \mathbb{S}_0^{2\times 2}.
\end{align*}
Under this coordinates, we can rewrite the relaxation set and wave-cone as
\begin{align*}
    &K_r=\{(z,\zeta):\left|z\right|^2=r\text{ and }\zeta=\frac{1}{2}z^2\},\\
    &\Lambda=\{(z,\zeta):\text{Im}(z^2\bar{\zeta})=0\}.
\end{align*}
Since $K_r$ and $\Lambda$ are invariant under the transformations
\begin{align*}
    R_{\theta}:(z,\zeta)\mapsto(ze^{i\theta},\zeta e^{2i\theta}),\ \theta\in[0,2\pi],
\end{align*}
and
\begin{align*}
    (z,\zeta)\mapsto(\bar{z},\bar{\zeta}).
\end{align*}
It's natural to consider the three-dimensional subspace
\begin{align*}
    L=\{(z,\zeta)\in\mathbb{C}\times\mathbb{C}:\text{Im}(\zeta)=0\},
\end{align*}
where we can use the coordinates $(a+ib,c)\in\mathbb{C}\times \mathbb{R}\cong L$. Then, we have
\begin{align*}
    K_r\cap L=\{(\sqrt{r},\frac{1}{2}r),(-\sqrt{r},\frac{1}{2}r),(i\sqrt{r},-\frac{1}{2}r),(-i\sqrt{r},-\frac{1}{2}r)\},
\end{align*}
and
\begin{align*}
    \Lambda\cap L=\{(a+ib,c):abc=0\}.
\end{align*}
\quad For fixed $r>0$, we define
\begin{align*}
    f_r(a+ib,c):=\frac{\sqrt{r}\left|a\right|}{\frac{r}{2}+c}+\frac{\sqrt{r}\left|b\right|}{\frac{r}{2}-c},\quad \left|c\right|<\dfrac{r}{2},
\end{align*}
and set
\begin{align*}
    V_r=\{(z,c)\in L:f_r(z,c)<1,\left|c\right|<\frac{r}{2}\}.
\end{align*}
Moreover, define
\begin{align}\label{eq11}
    \mathcal{V}_r:=\{(ze^{i\theta},ce^{2i\theta})\in \mathbb{C}\times\mathbb{C}:(z,c)\in V_r,0<\left|c\right|<\frac{r}{2},\theta\in\mathbb{R}\}\quad\text{and}\quad\mathcal{U}_r:=\mathcal{V}_r^{\text{lc}},
\end{align}
where the lamination convex hull $\mathcal{U}^{\text{lc}}$ of the set $\mathcal{U}$ with respect to the wave-cone $\Lambda$ is defined as
\begin{align}\label{eq12}
    \mathcal{U}^{\text{lc}}:=\mathop{\cup}\limits_{i=0}^{\infty}\mathcal{U}^{(i)},
\end{align}
where
\begin{align*}
    &\mathcal{U}^{0}=\mathcal{U},\\
    &\mathcal{U}^{(i+1)}:=\mathcal{U}^{(i)}\cup\{t\xi+(1-t)\xi':\xi,\xi'\in\mathcal{U}^{(i},\xi-\xi'\in\Lambda,t\in[0,1]\}.
\end{align*}
\quad Under this definition, we have the following Lemma:
\begin{lemma}(\cite[Corollary 16]{MR3505175})
The set $\mathcal{U}_r$ satisfied the perturbation Property (P) and \eqref{eq8}.
\end{lemma}

\section{Proof of Theorem 1.1}
\quad The proof of the Theorem \ref{tm1.1} is divided into two parts. For the first part, we consider the case in periodic space.
\begin{proposition}\label{pp6.1}
    Let $d\geq 2$, let $v_0$ be a smooth stationary Euler flow on $\mathbb{T}^{d}$. Given a smooth function $e(x)>\left|v_0(x)\right|^2$ for $x\in\mathbb{T}^{d}$. Then, for any $\sigma>0$, there exist infinitely many weak solutions $v\in L^{\infty}(\mathbb{T}^{d};\mathbb{R}^d)$ to \eqref{eq1} such that $\left|v(x)\right|^2=e(x)$ and $\left\|v-v_0\right\|_{H^{-1}}\leq \sigma$.
\end{proposition}
\begin{proof}
\textbf{Step 1: Functional setup.} Let $e=e(x)$ be a positive smooth function. Define
\begin{align*}
	X_0=\{w\in C^{\infty}(\mathbb{T}^{d};\mathbb{R}^{d}\times \mathbb{S}_0^{d\times d}):w\ \text{is a subsolution s.t. }w(x)\in\mathcal{U}_{e(x)}\text{ for all }x\in\mathbb{R}^{d}\}.
\end{align*}
\quad We claim that $X_0$ is bounded in $L^2$. In fact, let $\bar{e}=\underset{x\in\mathbb{R}^{d}}{\max}\ e(x)$. Notes that, for the set $K_{\bar{e}}^{\text{co}}=\{(v,\mathbf{U})\in\mathbb{R}^{d}\times\mathbb{S}_0^{d\times d}:v\otimes v-\mathbf{U}\leq\dfrac{e}{d}\mathbf{I}_d\}$, if $w(x)=(v(x),\mathbf{U}(x))\in\mathcal{U}_{e(x)}\subset K_{\bar{e}}^{\text{co}}$, by taking the trace, we have
\begin{align*}
	\left|v\right|^2\leq \bar{e},
\end{align*}
Therefore, $\|v\|_{L^{\infty}}^2\leq \bar{e}$. Furthermore, since $\mathbf{U}$ is symmetric and trace-free, the operator matrix norm can be estimated as
\begin{align*}
	\left|\mathbf{U}\right|_{\infty}\leq(d-1)\left|\lambda_{\min}(\mathbf{U})\right|,
\end{align*}
where $\lambda_{\min}(\mathbf{U})$ denote the smallest eigenvalue of $\mathbf{U}$. In turn, for any $\xi\in\mathbb{R}^d\backslash\{0\}$, we have
\begin{align*}
	-\xi^{T}\mathbf{U}\xi=\xi^{T}(v\otimes v-\mathbf{U})\xi-\left|v\xi\right|^2\leq\dfrac{\bar{e}}{d}.
\end{align*}
Therefore, $\|\mathbf{U}\|_{L^{\infty}}\leq \bar{e}$. 

\quad Take the divergence on both side of the equation \eqref{eq5}, one has
\begin{align*}
    \text{div }\text{div }\mathbf{U}+\Delta q=\text{div }\mathbf{B}v.
\end{align*}
Since $w=(v,\mathbf{U})\in C^\infty(\mathbb{T}^d)$, one can decompose the term $q$ as $q=q_1+q_2$ where
\begin{align*}
    \Delta q_1=\text{div }\mathbf{B}v\quad\text{and}\quad\Delta q_2=-\text{div }\text{div }\mathbf{U}.
\end{align*}
Note that, $v,\mathbf{U}\in C^\infty(\mathbb{T}^d)$. By standard elliptic estimate, $\|q\|_{L^2}\leq C\bar{e}$. Hence, $X_0$ is bounded in $L^{2}$. We define $X$ be the closure of $X_0$ in the weak $L^2$ topology, which is metrizable by boundedness.

\textbf{Step 2: $X$ contains smooth stationary flows.}
Let $v_0$ be a smooth solution of \eqref{eq1} with pressure $p_0$, and let $e=e(x)$ be a smooth function such that $e(x)>\left|v_0(x)\right|^2$ for all $x\in\mathbb{T}^{d}$. Let
\begin{align*}
	\mathbf{U}_0=v_0\otimes v_0-\dfrac{\left|v_0\right|^2}{d}\mathbf{I}_d,\ q_0=p_0+\dfrac{\left|v_0\right|^2}{d}.
\end{align*}
By definition, $w_0=(v_0,\mathbf{U}_0)$ is a subsolution and it satisfies
\begin{align*}
	w_0(x)\in K_{\left|v_0\right|^2(x)}\ \text{for all }x\in\mathbb{T}^{d}.
\end{align*}
By assumption \eqref{eq8}, $K_{\left|v_0(x)\right|^2}\subset\mathcal{U}_{e(x)}$ since $e(x)>\left|v_0(x)\right|^2$. Hence, $w_0(x)\in\mathcal{U}_{e(x)}$ for all $x\in\mathbb{T}^{d}$. Therefore,
\begin{align*}
    w_0\in X_0.
\end{align*}
\textbf{Step 3: Continuity points of $I(w):=\int\left|w\right|^2dx$.} The mapping $I(w):=\int\left|w\right|^2dx$ for $w\in X$ is a Baire-1 map in $X$. In fact, consider $I_{\varepsilon}(w)=\int\left|\rho_{\varepsilon}\ast w\right|^2dx$, $I_{\varepsilon}(w)$ are continuous functions and $I_{\varepsilon}(w)\rightarrow I(w)$ as $\varepsilon\rightarrow 0$ for all $w\in L^2$. Hence its continuity points form a residual set in $X$. On the other hand, Property (P) with an covering and rescaling argument leads to the following statements: there exists a strictly increasing continuous function $\bar{\Phi}:[0,+\infty)\rightarrow [0,+\infty)$ with $\bar{\Phi}(0)=0$ such that for any $w\in X_0$, there exists a sequence $w_k\in X_0$ such that
\begin{itemize}
	\item {$w_k\rightharpoonup w$ weakly in $L^2(\mathbb{T}^d)$;}
	\item {$\int_{\mathbb{T}^n}\left|w_k-w\right|^2dx\geq\bar{\Phi}(\int_{\mathbb{T}^n}\text{dist}(w(x),K_{e(x)}))$.}
\end{itemize}
\quad Consequently, using a diagonal argument and the metrizability of $X$, continuity points of the map $w\rightarrow \int\left|w\right|^2dx$ in $X$ are subsolutions $w$ such that $w(x)\in K_{e(x)}$ for almost every $x\in\mathbb{T}^n$. Indeed, assume that $I$ is continuous at $w\in X$ where $\int\text{dist}(w(x),K_{e})dx>\varepsilon$. Let $w_j\in X_0$ satisfy $w_j\rightharpoonup w$ with $\int\text{dist}(w_j(x),K_{e})dx>\varepsilon$. Let $\tilde{w}_j\in X_0$ such that $\tilde{w}_j-w_j\rightharpoonup 0$ (as $j\rightarrow \infty$) and $\int\left|\tilde w_j(x)-w_j(x)\right|^2dx>c(\varepsilon)$. Then, it contradicts with $\tilde{w}_j\rightarrow w$ (as $j\rightarrow \infty$) in $L^2$.

\quad Since a residual set in $X$ is dense, there exists a sequence $w_k=(v_k,\mathbf{U}_k)\in X$ with $w_k(x)\in K_{e(x)}$ almost everywhere such that $w_k\rightharpoonup w_0$. In particular, this means that $v_k$ is a weak stationary solution of \eqref{eq1}.

\end{proof}

\begin{proposition}\label{pp6.2}
    Let $d\geq 2$, let $\Omega$ be a bounded open subset on $\mathbb{R}^d$. Assume $e(x)\in C_c(\Omega;\overline{\mathbb{R}_{+}})$ is a non-trivial smooth function, then there exist infinitely many compact supported weak solutions $v\in L^{\infty}(\mathbb{R}^n)$ to \eqref{eq1} such that $\left|v(x)\right|^2=e(x)$ for a.e. $x\in \Omega$.
\end{proposition}

\begin{proof}
\textbf{Step 1: Functional setup.} Let $e=e(x)$ be a smooth function satisfying $\text{supp }e(x)\in \Omega$. Let $X_0$ be the set of elements $(v,\mathbf{U})\in C^{\infty}(\mathbb{R}^n)$ such that
\begin{itemize}
    \item {$\text{supp }(v,\mathbf{U})\subset \Omega$;}
    \item {$(v,\mathbf{U})$ is a subsolution to \eqref{eq1};}
    \item {$(v,\mathbf{U})(x)\in \mathcal{U}_{e(x)}$ for all $x\in \Omega$.}
\end{itemize}

\quad We claim that $X_0$ is bounded in $L^2$. In fact, since $e(x)\in C_c^\infty(\Omega)$, let $\bar{e}=\underset{x\in\mathbb{R}^{d}}{\max}e(x)$. Note that, if $w(x)=(v(x),\mathbf{U}(x))\in \mathcal{U}_{e(x)}\subset K_{\bar{e}}^{\text{co}}$, by taking the trace, we have
\begin{align*}
	\left|v\right|^2\leq \bar{e}.
\end{align*}
Therefore, $\|v\|_{L^{\infty}}\leq \sqrt{\bar{e}}$. Furthermore, since $\mathbf{U}$ is symmetric and trace-free, the operator matrix norm can be estimated as
\begin{align*}
	\left|\mathbf{U}\right|_{\infty}\leq(d-1)\left|\lambda_{\min}(\mathbf{U})\right|.
\end{align*}
In turn, for any $\xi\in\mathbb{R}^d\backslash\{0\}$, we have
\begin{align*}
	-\xi^{\text{T}}\mathbf{U}\xi=\xi^{\text{T}}(v\otimes v-\mathbf{U})\xi-\left|v\xi\right|^2\leq\dfrac{\bar{e}}{d}.
\end{align*}
Therefore, $\|\mathbf{U}\|_{L^{\infty}}\leq \bar{e}$. 

\quad Take the divergence on both side of the equation \eqref{eq5}, one has
\begin{align*}
    \text{div }\text{div }\mathbf{U}+\Delta q=\text{div }\mathbf{B}v.
\end{align*}
Since $w=(v,\mathbf{U})\in C^\infty(\Omega)$, one can decompose the term $q$ as $q=q_1+q_2$ where
\begin{align*}
    \Delta q_1=\text{div }\mathbf{B}v\quad\text{and}\quad\Delta q_2=-\text{div }\text{div }\mathbf{U}.
\end{align*}

Note that, both $q_1$ and $q_2$ are compact supported since $e\in C_c(\Omega)$ and $\text{supp}(v,\mathbf{U})\subset\Omega$. By standard elliptic estimate, $\|q\|_{L^2}\leq C\bar{e}$. Hence, $X_0$ is bounded in $L^2$. We define $X$ to be the closure of $X_0$ in the weak $L^2$ topology, which is metrizable by boundedness.

\textbf{Step 2: $X$ contains smooth stationary flows.}
By definition, $w_0=(0,\mathbf{0})$ is a subsolution, and it suffices to show that $w_0\in X_0$. By construction of the relaxation set, for $d\geq 3$, $\mathcal{U}_{e(x)}=\mathring{K}_{e(x)}^{\text{co}}$, and by definition, $w_0\in \mathring{K}_{e(x)}^{\text{co}}$. For $d=2$, though in \eqref{eq11}, the case $c=0$ is excluded in the definition of $\mathcal{V}_r$, it can be showed that under the induction \eqref{eq12}, $w_0\in \mathcal{U}_{e(x)}$ still holds.

\textbf{Step 3: Continuity points of $I(w):=\int\left|w\right|^2dx$.} The mapping $I(w):=\int\left|w\right|^2dx$ for $w\in X$ is a Baire-1 map in $X$. In fact, consider $I_{\varepsilon}(w)=\int\left|\rho_{\varepsilon}\ast w\right|^2dx$, $I_{\varepsilon}(w)$ are continuous functions and $I_{\varepsilon}(w)\rightarrow I(w)$ as $\varepsilon\rightarrow 0$ for all $w\in L^2$. Hence its continuity points form a residual set in $X$. On the other hand, property (P) with an covering and rescaling argument leads to the following statements: there exists a continuous strictly increasing function $\bar{\Phi}:[0,+\infty)\rightarrow [0,+\infty)$ with $\bar{\Phi}(0)=0$ such that for any $w\in X_0$, there exists a sequence $w_k\in X_0$ such that
\begin{itemize}
	\item {$w_k\rightharpoonup w$ weakly in $L^2(\mathbb{R}^d)$;}
	\item {$\int_{\mathbb{R}^d}\left|w_k-w\right|^2dx\geq\bar{\Phi}(\int_{\mathbb{R}^d}\text{dist}(w(x),K_{e(x)}))$.}
\end{itemize}
\quad Consequently, using a diagonal argument and the metrizability of $X$, continuity points of the map $w\rightarrow \int\left|w\right|^2dx$ in $X$ are subsolutions $w$ such that $w(x)\in K_{e(x)}$ for almost every $x\in\mathbb{R}^d$. Indeed, assume that $I$ continuous at $w\in X$ where $\int\text{dist}(w(x),K_{e})dx>\varepsilon$. Let $w_j\in X_0$ satisfying $w_j\rightharpoonup w$ with $\int\text{dist}(w_j(x),K_{e})dx>\varepsilon$. Let $\tilde{w}_j\in X_0$ such that $\tilde{w}_j-w_j\rightharpoonup 0$ (as $j\rightarrow \infty$) and $\int\left|\tilde w_j(x)-w_j(x)\right|^2dx>c(\varepsilon)$. Then, it contradicts with $\tilde{w}_j\rightarrow w$ (as $j\rightarrow \infty$) in $L^2$.

\quad Since a residual set in $X$ is dense, there exists a sequence $w_k=(v_k,\mathbf{U}_k)\in X$ with $w_k(x)\in K_{e(x)}$ almost everywhere such that $w_k\rightharpoonup w_0$. In particular, this means that $v_k$ is a weak stationary solution of \eqref{eq1}.
\end{proof}

Combining Propositions \ref{pp6.1} and \ref{pp6.2}, we complete the proof of the Theorem \ref{tm1.1}.

\textbf{Acknowledgements}

The author is grateful to his advisor Chunjing Xie for illuminating instruction and for several detailed comments which have shaped this paper into present form.
\bibliographystyle{abbrv}
\bibliography{references}

\end{document}